\documentclass[draft]{publmathdeb}
\usepackage{amsmath,amsfonts,amssymb}
\newcounter{alphthm}
\newtheorem{propriete}[alphthm]{Theorem}

\newtheorem{thm}{Theorem}[section]
\newtheorem{lem}[thm]{Lemma}
\newtheorem{cor}{Corollary}[section]

\theoremstyle{definition}
\newtheorem{definition}{Definition}[section]
\theoremstyle{remark}
\newtheorem{remark}{Remark}[section]

\author{B. Bidabad}
\address{Behroz Bidabad\\
Department of Mathematics and Computer Sciences\\
Amirkabir University of Technology (Tehran Polytechnic)\\
424 Hafez Ave. 15914 Tehran\\ Iran} \email{bidabad@aut.ac.ir}
\author{A. Tayebi}
\address{Akbar Tayebi\\
Department of Mathematics and Computer Sciences\\
Amirkabir University of Technology (Tehran Polytechnic)\\
424 Hafez Ave. 15914 Tehran\\ Iran}\email{akbar\_tayebi@aut.ac.ir}
\def\beq{\begin{equation}}
\def\eeq{\end{equation}}
\def\nn{\nonumber}
\title[A  classification of some Finsler connections]
      {A  classification of some Finsler connections\\ and their applications}
\begin{document}
\begin{abstract}
Some general Finsler connections are defined. Emphasis is being made
on the Cartan tensor and its derivatives. Vanishing of the
hv-curvature tensors of these connections characterizes
Landsbergian, Berwaldian as well as Riemannian structures. This view
point makes it possible to give a smart representation of connection
theory in Finsler geometry and yields  to a classification of
Finsler connections. Some practical applications of these
connections are also considered.
\end{abstract}
\maketitle
\section{Introduction}
There is always  a hope of finding a solution to some of the
unsolved problems of Finsler geometry by developing a connection
theory. This hope justifies the introduction of  new connections
\cite{BCS1}.  The study of hv-curvature of Finsler connections is,
by some authors, thought to be even urgent for theoretical physics,
see for instance \cite{KT}, \cite{MS} and \cite{T}. Vanishing
hv-curvatures of Berwald and Cartan connections characterize
Berwaldian and Landsbergian structures respectively \cite{Be},
\cite{Ca}. Discovery of Shen connection whose hv-curvature
characterizes the Riemannian structure, seems to completes their
works and permits the classification of Finsler connections into
three different categories \cite{Sh1}.

In this  paper, using the vanishing property of hv-curvatures, we
define three general kinds of Finsler connections and extend the
above property to a general family of Finsler connections. This
point of view enables us to define a more general family of  Finsler
connections which contains some known Finsler connections as special
cases. This characterization gives rise to the classification of
some Finsler connections with respect to the Cartan tensor and its
derivatives, which is a smart representation of Finsler connections
(see table of section 5). The distinguished property of this
connection is the flexibility of its reduced hv-curvature, which
makes it very useful. In fact its reduced hv-curvature may be chosen
to be equal to any linear differential equation formed in terms of
Cartan tensor and its derivatives. The above property makes the
geometric interpretation of the solutions of these differential
equations easy. As  application of this connection, we consider some
examples, especially those in which the flag curvature is constant.
\section{Preliminaries}
 Let $M$ be a n-dimensional $ C^\infty$ manifold.  $T_x M $ denotes
  the tangent space of M at $x$. The tangent bundle of M
 is the union of tangent spaces $TM:=\cup _{x \in M} T_x M$.  We will denote  the
elements of TM by $(x, y)$ where $y\in T_xM$. Let $TM_0 = TM
\setminus \{ 0 \}.$ The natural projection $\pi: TM_0 \rightarrow M$
is given by $\pi (x,y):= x$.

 A {\it Finsler structure} on  $M$ is a function $ F:TM
\rightarrow [0,\infty )$ with the following properties; (i) $F$ is
$C^\infty$ on $TM_0$,\  (ii) $F$ is positively 1-homogeneous on the
fibers of tangent bundle $TM$,  and  (iii) the Hessian of $F^{2}$
with elements $ g_{ij}(x,y):=\frac{1}{2}[F^2(x,y)]_{y^iy^j} $ is
positively defined on $TM_0$.  The pair  $(M,F)$ is then called a
{\it Finsler manifold.} $F$ is Riemannian if $g_{ij} (x,y)$ are
independent of $y \neq 0$.

Let us consider the pull-back tangent bundle $\pi^*TM$ over $TM_0$
defined by $ \pi^*TM=\left\{(u, v)\in TM_0 \times TM_0 |
\pi(u)\\=\pi(v)\right\} $. Take a local coordinate system $(x^i)$ in
M, the local natural frame $\{{{\partial} \over {\partial x^i}}\}$
of  $T_xM$ determines a local natural frame $\partial_i|_v$ for
$\pi^*_vTM$ the fibers of $\pi^*TM$, where  ${
\partial _i |_v=(v,{{\partial} \over {\partial x^i}}| _x )}$,
and $v=y^i{{\partial}\over {\partial x^i}}|_x\in TM_0$. The fiber
$\pi^*_vTM$ is isomorphic to
 $T_{\pi(v)}M$ where $\pi(v)=x$. There is a canonical section $\ell$
 of $\pi^*TM$ defined by $\ell_v=(v,v)/F(v)$.

Let $TTM$ be the tangent bundle of $TM$ and  $\rho$ the canonical
linear mapping $\rho:TTM_0\rightarrow \pi^*TM$ defined by
$\rho(\hat{X})=(z,\pi_{*}(\hat{X}))$ where $\hat{X}\in T_zTM_0$ and
$z\in TM_0$. The bundle map $\rho$ satisfies $ \rho ({\partial \over
{\partial x^i}})=\partial_i$ and $ \rho({\partial \over {\partial
y^i}})=0$.  Let $ V_zTM$ be the set of vertical vectors at $z$, that
is, the set of vectors tangent to the fiber through $z$, or
equivalently $V_zTM=ker \rho$,\  called  the \emph{vertical space}.

Let $\nabla$ be a linear connection on  $\pi^*TM$, that is $\nabla
:T_zTM_0\times \pi^*TM \rightarrow \pi^*TM$ such that $\nabla:
(\hat{X},Y) \rightarrow \nabla_{\hat{X}}Y$. Consider the linear
mapping $\mu_z:T_zTM_0\rightarrow T_{\pi z}M$ defined  by
$\mu_z(\hat{X})=\nabla_{\hat{X}}F\ell$,\ where $\hat{X}\in T_zTM_0$.
The connection $\nabla$ is called a \emph{Finsler connection} if for
every $z\in TM_0$, $\mu_z$ defines an isomorphism of $ V_zTM_0$ onto
$T_{\pi z}M$. Therefore, the tangent space $TTM_0$ in $z$ is
decomposed as $ T_zTM_0=H_zTM\oplus V_zTM$, where $ H_zTM=\ker\mu_z$
is called the {\it horizontal space} defined by $\nabla$.  Indeed
any tangent vector $\hat{X}\in T_zTM_0$ in $z$ decomposes to $
\hat{X}=H\hat{X}+V\hat{X}$ where $ H\hat{X}\in H_zTM$ and
$V\hat{X}\in V_zTM $.  The structural equations of the Finsler
connection $\nabla$ are
 \beq
{\mathcal T}_\nabla(\hat{X},\hat{Y})=
\nabla_{\hat{X}}Y-\nabla_{\hat{Y}}X-\rho[\hat{X},\hat{Y}],\label{torsion}
 \eeq
 \beq
\Omega(\hat{X},\hat{Y})Z=\nabla_{\hat{X}}\nabla_{\hat{Y}}Z-
\nabla_{\hat{Y}}\nabla_{\hat{X}}Z
-\nabla_{[\hat{X},\hat{Y}]}Z,\label{5}
 \eeq
where $X=\rho(\hat{X})$, $Y=\rho(\hat{Y})$ and $Z=\rho(\hat{Z})$.
The tensors ${\mathcal T} _\nabla$ and $\Omega$ are called
respectively the $Torsion$ and $Curvature$ tensors of $\nabla$. They
determine two torsion tensors defined by ${\mathcal S}(X,
Y):={\mathcal T} _\nabla(H\hat{X},H\hat{Y})$ and ${\mathcal
T}(\dot{X},Y):={\mathcal T} _\nabla(V\hat{X},H\hat{Y})$ and three
curvature tensors defined by
 $R(X,Y):=\Omega(H\hat{X},H\hat{Y})$, $
P(X,\dot{Y}):=\Omega(H\hat{X},V\hat{Y})$ and  $
Q(\dot{X},\dot{Y}):=\Omega(V\hat{X},V\hat{Y})$,\  where
$\dot{X}=\mu(\hat{X})$ and $\dot{Y}=\mu(\hat{X})$.

Given a Finsler structure F on M, then  at each point $x \in M$,
$F(v)=F(y^i{{\partial} \over {\partial x^i}}|_{x})$ is a function of
$(y^i) \in \mathbb{R}^n$. The \emph{fundamental tensor} $g$ is
defined by $g:\pi ^* TM \otimes \pi ^* TM \rightarrow [0,\infty)$
with the components $g(\partial_i |_v,\partial_j
|_{v})=g_{ij}(x,y)$. Thus  $(\pi ^* TM,g)$
 becomes a Riemannian vector bundle over
$TM_0$. The \textit{Cartan tensor} $A :\pi ^* TM \otimes \pi ^* TM
\otimes\pi ^* TM \rightarrow \mathbb{R} $ is defined by $
A(\partial_i |_v ,\partial_j |_ {v},\partial_k |_ {v})=A_{ijk}(x,y),
$ where  $ A_{ijk}(x,y)={1 \over 2}F(x,y)[F^2(x,y)]_{y^iy^jy^k}$. If
$A=0$ then  $F$ is  Riemannian.

{\bf Flag curvature}. A flag curvature is a geometrical invariant
that generalizes what in Riemannian geometry is called the sectional
curvature. For all $x \in M$ and  $0\neq y \in T_xM$, $V:=V^i
\frac{\partial}{\partial x^i}$ is called the transverse edge. Flag
curvature is obtained by carrying out the following computation at
the point $(x, y)\in TM_0$, and viewing $y$ and $V$ as sections of
$\pi^*TM$:
\[
K(y, V):=\frac{V^i(y^j \ R_{jikl}\ y^l) V^k}{g(y,y)g(V,V)-[g(y,
V)]^2}
\]
If $K$ is independent of the transverse edge $V$, then $(M,F)$ is
called of {\it scalar flag curvature}. Denoting this scalar by
$\lambda=\lambda(x, y)$, if it has no dependence on either $x$ or
$y$, then the Finsler manifold is said to be of {\it constant flag
curvature.}
\section{General-type Finsler connection}
In this section we define a general family of Finsler connections
which contains some known Finsler connections as special cases.
\begin{definition}
A tensor $S:\pi ^* TM \otimes \pi ^* TM \otimes \pi ^* TM
\rightarrow \mathbb{R}$ is called ``compatible" if it has the
following properties:\\
(1) $S(X, Y,Z)$ is symmetric with respect to $X$, $Y$, $Z$. \\
(2) $S(X, Y,\ell)=0$.\\
(3) $S$ is homogeneous, i.e., $S_{ijk}(x,ty)=S_{ijk}(x,y), \forall t
\in \mathbb{R}$, where $S_{ijk}(x, y)=S(\partial_{i},
\partial_{j}, \partial_{k})$.
\end{definition}
\begin{definition}
Consider a Finsler connection  $D$ on $(M,F)$. Let
 $S$ and $T$ be two compatible tensors on $\pi ^* TM$.\\
 (i) The {\it torsion} tensor ${\mathcal T}_D$ of $D$, defined by
(\ref{torsion}),
 should satisfy
\begin{eqnarray}
{\mathcal T} _D (\hat{X}, \hat{Y})&=&F^{-1}T(\mu(\hat{X}),\rho
(\hat{Y})))-F^{-1}T(\mu(\hat{Y}),\rho(\hat{X})),\label{1}
 \end{eqnarray}
 where $T(X,Y)$ is  defined by $g(T(X, Y), Z):= T(X,Y,Z)$,\ \
  $ \hat{X},\hat{Y} \in T_zTM_0$.\\
 (ii) Let
$ (D_{\hat{Z}} g)(X,Y):= \hat{Z}g(X,Y) -g(D_{\hat{Z}} X, Y) -g(X,
D_{\hat{Z}} Y). $ Then the connection $D$ is called  {\it
almost-compatible} with the Finsler structure if for all $X,Y \in
\pi ^* TM$ and $\hat{Z} \in T_zTM_0$,
\begin{eqnarray}
\nonumber(D_{\hat{Z}} g)(X,Y)&=&2
  A(\rho(\hat{Z}),X,Y)+2F^{-1}
A(\mu(\hat{Z}),X,Y)\\
&-&2S(\rho(\hat{Z}),X,Y)-2F^{-1}
T(\mu(\hat{Z}),X,Y).\label{2}
 \end{eqnarray}
    (iii) $D$ is called {\it metric-compatible} with Finsler structure if
$ (D_{\hat{Z}} g)(X,Y)=0.$
\end{definition}
For torsion-free connections the bundle map
 $\mu$ satisfies
 $\mu({\partial \over \partial y^i})=\partial_i$
and $\mu ({{\partial} \over {\partial x^i}})=N^k _i \partial_k$,
    where $N^k _i =F \Gamma ^k _{ij} \ell ^j$ and $\Gamma ^k _{ij}$ are Christoffel
symbols of the torsion-free Finsler connection $D$.

We have the following general theorem of existence and uniqueness of
linear connections in different versions.

\begin{propriete} (\cite{Sh1}) Let $(M,F)$ be a Finsler manifold. Suppose
 $S$ and $T$ are two compatible tensors in $\pi ^* TM$.
Then there exists a unique  almost-compatible linear connection $D$
with  torsion ${\mathcal T} _D$ on $\pi ^* TM$ satisfying
\textit{(i)} and \textit{(ii)}.\label{A}
\end{propriete}
Let $\bar \ell$ denote the unique vector field in $HTM$
  such that $\rho(\bar \ell)=\ell$. We define $\dot A,..., \overset{_{m+1}}{A}$
 from   $\pi ^* TM \otimes \pi ^* TM
\otimes \pi ^* TM $  to $ \mathbb{R}$ as follows:
\begin{equation}
\nonumber\dot A (X, Y, Z):= \bar \ell A(X,Y,Z) - A(D_{\bar \ell} X,
Y, Z)-
 A(X,D_{\bar \ell} Y, Z)- A(X,Y,D_{\bar \ell} Z),
 \eeq
 \beq
  \overset{_{m+1}}{A}(X, Y, Z):= \bar \ell\ \overset{_m}{A}(X,Y,Z) -\overset{_m}{A}
 (D_{\bar \ell} X,
Y, Z)- \overset{_m}{A} (X,D_{\bar \ell} Y, Z) - \overset{_m}{A}
(X,Y,D_{\bar \ell} Z),\label{3}
\end{equation}
where $\overset{{\ _0}}{A}:=A$, $\overset{\ _1}{A}:=\dot A$,
$\overset{\ _2}{A}:=\ddot A$,... and $m \in\mathbb{N}$. Obviously,
$\forall m\in\mathbb{N}$, the tensors $\overset{\ _m}{A}$ are
symmetric with respect to $X, Y$ and $Z$. Moreover, using $D_{\bar
\ell}\ \ell=0 $ we have $\overset{\ _m}{A}(X,Y,\ell)=0$.  A Finsler
metric is called a {\it {Berwald metric}} if  for
 any standard local coordinate system
$(x^i,y^i)$ in $TM_0$, the Christoffel symbols $\Gamma^k
_{ij}=\Gamma^k _{ij}(x)$ are functions of $x \in M$ alone. A Finsler
metric is called a {\it{Landsberg metric}} if $\dot A=0$.

By mean of Theorem A,  we can define the general Finsler connection.
 \begin{definition}
Let $(M,F)$ be a Finsler manifold. A { \it general-type Finsler
connection} is defined as  a Finsler connection $D$ on $\pi ^* TM$
such that  its compatible tensors $S$ and $T$ can be defined as
follows:
 \beq S:=
\kappa_{_0}A+\kappa_{_1}\dot A+\kappa_{_2}\ddot A+
\cdots+\kappa_{_m}\overset{_m}{A} \quad and \quad T:= r A,\label{4}
\eeq where the coefficients $\kappa_{_i}$, $i=1,...,m$ and $r$ are
real constants.
 \end{definition}
 {\section{Curvature Tensors}}
 Let $D$ be a Finsler connection defined on $M$. Let $\{e_i\}^n _{i=1}$ be a
   local orthonormal (with respect to $g$)
frame field for the vector bundle $\pi ^* TM$ such
  that $g(e_i ,e_n )=0, i=1,...,n-1$ and $e_n=\ell^i{\partial \over {\partial
  {x^i}}}$. Let $\{\omega^i\}^n_{i=1}$ be its dual co-frame field. One
readily finds that $ \omega^n:= {\partial{F} \over {\partial
  {y^i}}}dx^i=\omega$, which is called {\it Hilbert form}, and
   $\omega(\ell)=1$. Let
  $\rho =\omega^i \otimes e_i$, $De_i = \omega ^{\ j} _i \otimes
e_j$ and $ \Omega e_i=2\Omega ^{\ j} _i \otimes e_j$,
  where $\{\Omega ^{\ j} _{i}\}$ and $\{\omega ^{\ j} _{i}\}$ are called respectively,
   the  {\it curvature forms} and {\it connection forms} of $D$ with respect to
  $\{e_{i}\}$.  We have $\mu:=DF\ell=F\{\omega ^{\ i} _n +d(log F)\delta ^i _n
\}\otimes e_i.$ Put $\omega^{n+i}:=\omega^{\ i} _n  +d(log F)\delta
^i _n.$ It is easy to show that $\{\omega ^i, \omega^{n+i} \}^n
_{i=1}$ is a
  local basis for $T^*( TM_0).$
 The equation (\ref{5}) is equivalent to
   \begin{equation}
   d\omega ^{\ j} _i -\omega ^{\ k}_i \wedge \omega^{\ j}_ k=\Omega^{\
j}_ i.\label{6}
   \end{equation}
Since the $\Omega ^{\ i} _j$ are 2-forms on $TM_0$,  they can be
expanded as
  \begin{equation}
\Omega^{\ j}_ i={1 \over 2}R^{\ j} _{i \ kl} \omega ^k \wedge
\omega^l +P^{\ j} _{i \ kl} \omega ^k \wedge \omega^{n+l}+{1 \over
2}Q^{\ j} _{i \ kl} \omega ^{n+k} \wedge \omega^{n+l}.\label{7}
  \end{equation}
  Let $\{\bar e_i, \dot e_i\}^n _{i=1}$ be the local basis
  for $T(TM_0)$, which is dual to $\{\omega ^i, \omega^{n+i} \}^n
  _{i=1}$, i.e., $\bar e_i \in HTM, \dot e_i \in VTM$ such that
  $\rho(\bar e_i)=e_i, \mu(\dot e_i)=F e_i$.
The objects $R$, $P$ and  $Q$ are called, respectively,  the hh-,
hv- and vv-curvature  tensors of the connection $D$ with the
components $R(\bar e_k,\bar e_l)e_i =R^{\ j}
   _{i \ kl}e_j , \quad P(\bar e_k,\dot e_l)e_i=P^{\ j}_{i \ kl} e_j
\quad$   and $ \quad Q(\dot e_k,\dot e_l)e_i=Q^{\ j} _{i \ kl} e_j.$
  From (\ref{7}) we see that $R^{\ j} _{i \ kl}=-R^{\ j} _{i \ lk}$ and
  $ Q^{\ j} _{i \  lk}=-Q^{\ j} _{i \ kl}$. Let we put
  \begin{equation}
    dg_{ij}-g_{kj} \omega ^{\ k} _i-g_{ik } \omega ^{\ k} _j =
g_{ij| k}\omega ^{k}+ g_{ij.k}\omega ^{n+k},\label{9}
    \end{equation}
  \begin{equation}
    dA_{ijk}-A_{ljk} \omega ^{\ l} _i - A_{ilk}\omega ^{\ l} _j
-A_{ijl}\omega ^{\
    l} _k =A_{ijk | l}\omega ^l +A_{ijk.l}\omega ^{n+l},\label{10}
    \end{equation}
    where the slash "$|$ " and point "$.$ " are horizontal and vertical covariant
    derivatives with respect to the Finsler connection. In a similar
    way $\forall m\in\mathbb{N}$,  we have:
    \begin{equation}
    d \overset{_m}{A}_{ijk}-\overset{_m}{A}_{ljk}\omega ^{\ l} _i-\overset{_m}{A}_{ilk}
    \omega
^{\ l} _j-\overset{_m}{A}_{ijl}\omega ^{\ l}_k
=\overset{_m}{A}_{ijk|l}\omega ^l +\overset{_m}{A}_{ijk.l}\omega
^{n+l},\label{11}
    \end{equation}
 where $\overset{_m}{A}_{ijk}=\overset{_m}{A}( e_i,e_j,e_k)$ and
$\overset{_m}{A^k}_{ij}=g^{kl} \overset{_m}{A}_{ijk} $. From
(\ref{10}) and (\ref{11}), we see that $ A_{ijk|l}$, $A_{ijk.l}$,
 $\overset{_m}{A}_{ijk|l}$ and $\overset{_m}{A}_{ijk.l}$,
  ($\forall m,l\in\mathbb{N}$), are all symmetric with respect to i, j and  k.
  By definition of Landsberg tensor, we have $A_{ijk |n}= \dot A_{ijk}$.
   Here we use the notation
$\overset{_m}{A}_{ijk |n}=\overset{_m}{A}_{ijk |s}\ell^s$  and
 $\overset{_m}{A}_{ijk|n}= \overset{_{m+1}}{A}_{ijk}$. From (\ref{10}) and (\ref{11}),
  we get
   \begin{equation}
   A_{njk | l}=0,\ \  A_{njk. l}=-A_{jkl},\ \ \overset{_m}{A}_{njk
   |l}=0\  \  \textrm{and}\ \
    \overset{_m}{A}_{njk
.l}=-\overset{_m}{A}_{jkl}.\label{12}
\end{equation}
 \setcounter{thm}{0}
\begin{remark}
In general-type connection, the horizontal  and vertical covariant
derivatives of the metric tensor are given by
\[
g_{ij|k}=2((1-\kappa_{_0})A_{ijk}-\kappa_{_1}\dot
A_{ijk}+\cdots+\kappa_{_m}\overset{_m}{A}_{ijk})\ \  \textrm{and } \
\ g_{ij.k}= 2(1-r)A_{ijk}.
\]
\end{remark}
\section{A classification of some Finsler connections}
The following results due to Berwald, Cartan and Shen, determine the
relation between hv-curvature and special Finsler spaces. These
results enable us to classify some non-Riemannian Finsler
connections and distinguish  three different categories.
\begin{propriete} ({\cite{Be}, \cite{Ch}}) Let $(M,F)$ be a Finsler manifold.
Then for the Berwald connection (or Chern connection), hv-curvature
vanishes if and only if F is a Berwald metric.\label{B}
\end{propriete}
 \begin{propriete} ({\cite{Ca}}) Let $(M,F)$ be a Finsler manifold.
  Then for the Cartan connection (or Hashiguchi
connection), hv-curvature vanishes if and only if F is a Landsberg
metric.\label{C}
\end{propriete}
\begin{propriete} ({\cite{Sh1}}) Let $(M,F)$ be a Finsler manifold.
Then for the Shen connection, hv-curvature vanishes if and only if F
is Riemannian.\label{D}
\end{propriete}
The remarkable property of Shen connection,  proved by Theorem D,
comes from the fact that vanishing of  its hv-curvature singles out
Riemannian metrics. In contrast, Cartan, Berwald, Chern and
Hashiguchi connections do not possess this property. Thus we have
 three different types of Finsler connections. Theorems \ref{thm1},
 \ref{thm2} and \ref{thm3} of this paper, deal with  a more general
case and give rise to new families of Finsler connections that we
call Berwald-type, Cartan-type and Shen-type connections which are
defined according to the behavior of their hv-curvature.
\begin{definition} Let $(M,F)$ be a Finsler manifold. A Finsler
connection is called of {\it Berwald-type} (resp. {\it of
Cartan-type} or {\it Shen-type}) if and only if  vanishing of its
hv-curvature,  reduces the Finsler structure  to the Berwaldian
(resp. Landsbergian or Riemannian) one.
\end{definition}
From this  view point one can compare some of the non-Riemannian
Finsler connections according to the compatibility of the tensors
$S$ and $T$.
\begin{center}  \textbf{ A classification of Finsler
connections  according  to their compatible tensors $S$ and
$T$}\end{center} {\small
 \centerline
{\begin{tabular} {||l||c||c||c||c||} \hline\hline
&\multicolumn{2}{|c||}{\textbf{Compatible tensors}}&\ &
\\ \hline \textbf{Connection}& \textbf{S}&\textbf{T}& \textbf{Metric compatibility} & \textbf{Torsion}\\
 \hline \hline\hline 1. Berwald & $A+\overset{\bullet}{A}$ &
  \quad $0$& almost compatible&free\\
\hline 2. Chern- Rund &$A$ &\quad$0$&  almost compatible& free \\
\hline  3. Berwald-type & $A+\kappa_{_1}
\overset{\bullet}{A}+\cdots+\kappa_{_m}\overset{_m}{A}$ &\quad $0$&
almost
 compatible&  free\\
\hline \hline \hline 4.  Cartan & $A$ & \quad$A$& metric compatible&
not free  \\
\hline  5.  Hashiguchi &$A+\overset{\bullet}{A}$ & \quad$A$& almost
compatible& not free\\
\hline  6. Cartan-type &$A+\kappa_{_1}
\overset{\bullet}{A}+\cdots+\kappa_{_m}\overset{_m}{A}$  \ \ \ \ \ \
\  &\ \ \
$A$& depends on $\kappa_{_i}$& not free\\
\hline \hline\hline  7. Shen & $0$ & \quad$0$& almost compatible& free\\
\hline  8. Shen-type& $\kappa_{_1}
\overset{\bullet}{A}+\cdots+\kappa_{_m}\overset{_m}{A}$ & \quad$0$&
almost compatible& free\\
\hline \hline \hline 9. General-type  &$\kappa_{_0} A+\kappa_{_1}
\overset{\bullet}{A}+\cdots+\kappa_{_m}\overset{_m}{A}$ &\quad$r A$&
depends on
$\kappa_{_i}$ and $r$& \footnotesize{depends on $r$}\\
\hline
\end{tabular}}}
\bigskip
 In this table $A$, $\dot A$, $\ddot A$,...,
$\overset{_m}{A}$ are Cartan tensor and their covariant derivatives,
$\kappa_{_i}$ and $r$ are arbitrary real constants. The connections
1, 2, and 3 belong to the Berwald-type category. The connections 4,
5, and 6 are Cartan-type connections.
 The connections 7 and 8 belong to the Shen-type Category. The
 connection 9 contains all other connections. Looked at the freeness of torsion  point of
 view, the  Shen connection is the one most similar to
  the Levi-Civita connection. But from the metric compatibility view
  point, it is the Cartan connection which is closest  to
  the Levi-Civita connection.

Now we extend Theorem $C$ to Cartan-type connections and show that
the hv-curvature tensor of this type of connections
 characterizes Landsbergian structures.
\begin{thm}
Let $(M,F)$ be a Finsler manifold. Then for  Cartan-type
connections, hv-curvature vanishes if and only if F is
Landsbergian.\label{thm1}
\end{thm}

To prove Theorem \ref{thm1},  we need the following lemma.
\setcounter{thm}{0}
 \begin{lem}
Let $(M,F)$ be a Finsler manifold. Then for  Cartan-type
connections we have\\
  $ 1) \quad R^{\ i}_{j\ kl}+R^{\ i}_{k\ lj}+R^{\ i}_{l\ jk}=
  C^i_{jm}R^{\ m}_{n\ kl}+C^i_{km}R^{\ m}_{n\ lj}
   +C^i_{lm}R^{\ m}_{n \ jk},\quad
    $\\
 $
   2) \quad P^{\ i}_{j\ kl}=P^{\ i}_{k\ jl}+ C^i_{kl | j}-C^i_{jl | k}+
   C^i_{jr}P^{\ r}_{n\ kl}-C^i_{kr}P^{\ r}_{n\ jl},\quad \quad
   $\\
 $
  3) \quad Q^{\ i}_{j\ kl}=Q^{\ i}_{j\ lk}+2( C^i_{jk .l}-C^i_{jl .k})+
   2(C^i_{mk}C^m_{jl}-C^i_{ml}C^m_{jk})+C^i_{jm}(Q^{\ m}_{n\ kl}
   -Q^{\ m}_{n\ lk}),
  $\\
  where $R_{ijkl}=g_{sj}R^{\ s}_{i\ kl}$, $P_{ijkl}=g_{sj}P^{\ s}_{i\ kl}$,
   $Q_{ijkl}=g_{sj}Q^{\ s}_{i\ kl}$, $C^i_{jk}=F^{-1}g^{im}A_{mjk}$.
   \label{lem}
\end{lem}
\begin{proof} Let's consider the Cartan-type connection with compatible tensors $S=
A+\kappa_{_1} \dot A+\cdots+\kappa_{_m}\overset{_m}{A} $ and $T=A$.
Following (\ref{1}) and (\ref{2}),  there exits a
 connection 1-forms $\{\omega^i_j\}$  satisfying the following
torsion and almost compatibility conditions.
 \begin{equation}
   d\omega ^i =\omega ^j \wedge \omega^{\ i} _j-C^i_{kl}\ \omega ^k \wedge
   \omega^{n+l},\label{13}
 \end{equation}
   \begin{equation}
dg_{ij}=g_{kj} \omega ^{\ k} _i +g_{ik } \omega ^{\ k} _j -
 2\kappa_{_1}\dot A_{ijk}\omega ^k-\cdots-
 2\kappa_{_m}\overset{_m}{A}_{ijk}\omega ^k.\label{14}
  \end{equation}
 Differentiating (\ref{13}) and using (\ref{6}) and (\ref{10}), we get:
\[
\omega^j \wedge
\Omega^i_j=(C^i_{kl|j}\omega^j+C^i_{kl.j}\omega^{n+j})\wedge
 \omega^k \wedge \omega^{n+l}-C^i_{lm}C^m_{jk} \omega^j \wedge \omega^{n+k}
 \wedge \omega^{n+l}-C^i_{kl}\omega^k
 \wedge \Omega^l_n.
\]
Replacing $\Omega^i_j$ by (\ref{7}),  we prove the Lemma.
 \end{proof}
 \noindent{\it  \emph{\textbf{Proof of Theorem \ref{thm1}}}}:
Let $(M,F)$ be a Finsler manifold with
 Cartan-type connection and compatible tensors  $S=A+\tilde{S}$ and $T=A$,
  where $\tilde{S}=\kappa_{_{1}}\dot A+\cdots+
 \kappa_{_m}\overset{_m}{A}$. Then the almost compatibility condition (\ref{14})
 becomes
   \begin{equation}
dg_{ij}=g_{kj} \omega ^{\ k} _i +g_{ik } \omega ^{\ k} _j -
 2\widetilde{S}_{ijk}\omega ^k.\label{15}
  \end{equation}
Differentiating this relation leads to
 \[
g_{kj} \Omega^{\ k}_i+g_{ik}\Omega^{\ k}_j=2(\widetilde{S}_{ijk | s}
\omega ^s+\widetilde{S}_{ijk . s} \omega ^{n+s} )\wedge
\omega^k+2C^i_{uv}\widetilde{S}_{ijk}\omega ^{u} \ \wedge \omega
^{n+v}.
  \]
From this relation and (\ref{7}) we have
  \begin{equation}
  R_{ijkl}+R_{jikl}=2(\widetilde{S}_{ijl|k}-\widetilde{S}_{ijk|l}),\quad
  \quad\label{16}
  \end{equation}
   \begin{equation}
   P_{ijkl}+P_{jikl}=-2(\widetilde{S}_{ijk.l}-C^u_{kl}\widetilde{S}_{uij}),\label{17}
    \end{equation}
  \begin{equation}
Q_{ijkl}+Q_{jikl}=0.\quad \quad \quad\quad \quad \quad \quad
\quad\label{18}
   \end{equation}
 Permuting $i,j$  and $k$ in (\ref{17}) and using Lemma 1
 yields
   \[
    P_{ijkl}=-\widetilde{S}_{ijk.l}+(C^u_{kl}\widetilde{S}_{uij}+
    C^u_{il}\widetilde{S}_{ujk}-C^u_{jl}\widetilde{S}_{uki})+
     (C_{kjl|i}-C_{ijl|j})
 \]
 \begin{equation}
  \hspace{6.5cm}+(C_{kiv}P^v_{njl}-C_{jkv}P^v_{nil}).\label{19}
  \end{equation}
Multiplying this relation by $y^i$ and replacing
$\widetilde{S}=\kappa_{_1} \dot
A+\cdots+\kappa_{_m}\overset{_m}{A}$,  we get
\begin{equation}
P_{njkl}=\dot C_{ijk}+\{\kappa_{_{1}}\dot
A_{ijk}+\cdots+\kappa_{_m}\overset{_m}{A}_{ijk}\}.\label{21}
\end{equation}
If $F$ is a Landsbergian manifold, from the above relation, we have
$P_{njkl}=0$. Therefore by replacing this value in (\ref{19}) we
find $ P_{ijkl}=C_{kjl|i}-C_{kil|j}$. In the case of Landsbergian
manifolds, $C_{ijk|l}$ is totally symmetric in all of its four
indices and we have $ P_{ijkl}=0 $. Conversely, let hv-curvature be
zero.  Then by Lemma 1,  we have $ C^i_{kl | j}=C^i_{jl | k}$,
therefore $M$ is Landsbergian. \ \ \hspace{4.5cm}$\Box$
\begin{thm}
Let $(M,F)$ be a Finsler manifold. Then for  Berwald-type
connections, hv-curvature vanishes if and only if F is a Berwaldian
metric.\label{thm2}
\end{thm}
 \begin{proof} The complete proof of this theorem, will not be given,
 but only a sketch of the proof will be presented. For a Berwald-type connection,
  the hv-curvature is
      \begin{eqnarray} P_{ijkl}&=&-\{\kappa_{_1}\dot A_{ijk.l}+\cdots+
     \kappa_{_m}\overset{_m}{A}_{ijk .l}\}
     -(A_{ijl | k}+A_{jkl | i}-A_{kil | j})\nn \\&&+A_{kis}
P^{\ s} _{n \ jl}-A_{jks} P^{\ s} _{n \ il}-A_{ijs} P^{\ s} _{n\
kl}.\label{23}
  \end{eqnarray}
Therefore, we have
\begin{equation}
P_{njkl}=\{\kappa_{_1}\dot A_{jkl}+\cdots+
\kappa_{_m}\overset{_m}{A}_{jkl}\}-\dot A_{jkl}.\label{24}
\end{equation}
Using these relations,  the theorem will follow.
\end{proof}

\begin{thm} Let $(M,F)$ be a Finsler manifold.
Then for  Shen-type connections, hv-curvature vanishes  if and only
if F is Riemannian.\label{thm3}
\end{thm}
\begin{proof} The proof of this theorem is analogous to
that of the Theorem 5.1 and is not presented here.
\end{proof}

\begin{thm} Let $(M,F)$ be a Finsler manifold.
Then the hv-curvature of  general-type (respectively Berwald-type,
Cartan-type or Shen-type) connections vanishes if and only if F is
Berwaldian, Landsbergian  or Riemannian.\label{thm4}
\end{thm}

{\section{Some applications of general-type connections}} Much of
the practical importance of this kind of connection results from the
fact that, it is adaptable, in the sense that it is useful for
getting a geometric interpretation for a given system of
differential equation formed by Cartan tensor and its derivatives.
Suppose that we are given a differential equation of this kind and
we want to find  a geometric meaning for its solutions. It would
suffice to consider a Finsler connection -- by fixing the compatible
tensors $S$ and $T$ -- for which the reduced hv-curvature coincides
with the differential equation in question. We then apply one of the
Theorems \ref{thm1}, \ref{thm2} or \ref{thm3} as applicable.
\subsection{Application of Shen-type connections}

Here  we define  Shen-type connection $D$ as
 $S_{ijk}=(1-k)A_{ijk}+k\dot A_{ijk}-\ddot A_{ijk}$ and $T_{ijk}=0$
for which the reduced hv-curvature  $P_{jkl}:=\ell^i \ P_{ijkl} $ is
equal to the given differential equation $P_{jkl}=\ddot
A_{jkl}+kA_{jkl}$.
\begin{thm} Let (M,F) be a  Finsler manifold
with constant flag curvature $\lambda$ such that $P_{jkl}=0$. Then
$F$ is Riemannian.\label{thm5}
\end{thm}
\begin{proof} Let's  consider the Shen-type connection $D$ with
$S_{ijk}=(1-k)A_{ijk}+k\dot A_{ijk}-\ddot A_{ijk}$, $k\neq\lambda$
and $T=0$. Replacing $S$ and $T$ in (\ref{2}) and by an argument
similar to  the one used in the proof of the Theorem 1, we get
   \begin{equation}
   P_{ijkl}+ P_{jikl}=-2\{\dot A_{ijk.l}+\ddot A_{ijk.l}\}
   -2k\{A_{ijk.l}-A_{ijl| k}\}-2A_{ijm} P^{\ m}_{n \ kl}.\label{5-5}
   \end{equation}
 From (\ref{5-5}) we have
   \begin{eqnarray}
 \nonumber  P_{ijkl}&=&-\{\dot A_{ijk.l}+\ddot A_{ijk.l}\}-kA_{ijk.l}
   +k\{-A_{ijl| k}+A_{jkl| i}-A_{kil| j}\}\\ &-&2A_{ijm}P^{\ m}_{n \ \ kl}
   +2A_{kim} P^{\ m}_{n \ \ jl}-2A_{jkm}P^{\ m}_{n \ \ il}.\label{5-7}
    \end{eqnarray}
Therefore
 $
P_{jkl}=\ddot A_{jkl}+kA_{jkl}$. The equation $P_{jkl}=0$ holds,
from which we have
 \begin{equation}
 \ddot A+kA=0.\label{5-9}
\end{equation}
Since $(M,F)$ is a Finsler manifold with constant flag curvature
$\lambda$,  then
\begin{equation}
\ddot A+\lambda A=0.\label{5-10}
\end{equation}
From (\ref{5-9}) and (\ref{5-10}) one has $(\lambda-k)A=0$ which
means that  $F$ is a Riemannian metric.
\end{proof}
Using  the above special Shen-type connection again together with a
hypothesis on the topology of $M$,  we have the following theorem.
 \begin{thm} Let (M,F) be a complete Finsler manifold with bounded Cartan
 tensor. Then  $(M,F)$
 is a Riemannian manifold if and only if $P_{jkl}=0$.\label{thm6}
\end{thm}
\begin{proof} Let's consider the
above Shen-type connection on the complete Finsler manifold (M,F).
Then from the last theorem we have that the hv-curvature of this
connection reduces to $P_{jkl}=\ddot A_{jkl}+kA_{jkl}$. Fix any $X,
Y, Z \in \pi^*TM$ at $v \in I_xM=\{w \in T_xM,F(w)=1\}$. Let $c:
\mathbb{R}\rightarrow M$ be the unit speed geodesic in $(M,F)$ with
${{dc} \over {dt}}(0)=v$ and  $\hat c:={dc \over{dt}}$ be the
canonical lift of $c$ to $TM_0$. Let $X(t)$, $Y(t)$ and $Z(t)$
denote the parallel sections along $\hat c$ with $X(0)=X$, $Y(0)=Y$
and $Z(0)=Z$. Put $A(t)=A(X(t),Y(t),Z(t))$, $\dot A(t)=\dot
A(X(t),Y(t),Z(t))$ and $\ddot A(t)=\ddot A(X(t),Y(t),Z(t))$. Indeed
along geodesics, we have $ \frac{d \dot A}{dt}=\ddot A$ and from
$\ddot A_{jkl}+kA_{jkl}=0$, we get
\begin{equation}
A(t)=(c_1 \sinh \sqrt{k}t+c_2 \cosh
\sqrt{k}t)A(0).\label{5-13}
\end{equation}
For $v \in TM_0$, let's define $\| A \|_v:=sup A(X,Y,Z)$  where the
supremum is taken over all unit vectors of $ \pi_v ^* TM$. Let's put
$\| A \|=sup_{v \in IM}\| A\|_v$  where $IM=\bigcup_{x \in M}I_xM$.
Since $M$ is complete and $\|A\|<\infty$, by letting  $t\rightarrow
+\infty$ and $t\rightarrow -\infty$, we have $c_1=0$ and $c_2=0$.
Therefore  $A=0$, and  $F$ is  Riemannian.
\end{proof}

\subsection{Application of Berwald-type connections.}
Here we consider a special Berwald-type connection for which
 the hv-curvature is equal to the given  differential equation.
\begin{thm}
 Let (M,F) be a complete Finsler manifold with bounded Landsberg tensor.
Then $F$ is a Landsberg metric  if and only if
$P_{jkl}=0$.\label{thm7}
\end{thm}
\begin{proof} If we put $\kappa_{_1}=\kappa_{_3}=\cdots=\kappa_{_m}=0$ and
$\kappa_{_2}\neq0$ in (\ref{24}) then we find a special Berwald-type
connection for which the hv-curvature is equal to
$P_{jkl}=\kappa_{_2}\ddot A_{jkl}-\dot A_{jkl}$. Let F be a
Landsberg metric, then from the above equation we get $P_{jkl}=0$.
Conversely,  if $P_{jkl}=0$, we will have:
\begin{equation}
\kappa_{_2}\ddot A_{jkl}-\dot A_{jkl}=0.\label{5-1}
\end{equation}
By an argument like the one presented in the last theorem, we have
along the geodesics
 \beq
  \dot A(t)=e^{k_{_{\text 2}}t}\dot
A(0).\label{5-*}
 \eeq
 For $v \in TM_0$, let's define $\|\dot A \|_v:=sup \dot A(X,Y,Z)$ and
   $\|\dot A\|=sup_{v \in IM}\|\dot A\|_v$. Using  completeness of $M$,
   $\|\dot A\| < \infty$   and letting  $t \rightarrow  +\infty$ we
   have  $ \dot A(0)=\dot A(X,Y,Z)=0$. From (\ref{5-*}),  we get
$\dot A=0$, that is, $F$ is a Landsberg metric.\
\end{proof}
  \setcounter{thm}{0}
\begin{cor}  Every compact Finsler manifold
is Landsbergian if and only if $P_{jkl}$ vanishe.
 \end{cor}
 Next we consider another special Berwald-type connection and give a proof
of the following well-known result due to Akbar-Zadeh \cite{AZ}.
\begin{cor}.
Let (M,F) be a complete Finsler manifold with negative constant flag
curvature $\lambda$ and bounded Cartan tensor. Then $F$ is
Riemannian.
\end{cor}
\begin{proof} Let's put $\kappa_{_2}=\kappa_{_4}=\cdots=\kappa_{_m}=0$ \ , $
\kappa_{_1}=2$ and $\kappa_{_3}=\frac{1}{\lambda} \neq0$ in
(\ref{23}). We obtain a connection for which the hv-curvature
becomes
 \begin{eqnarray} P_{ijkl}&=&-\{2\dot A_{ijk.l}+
     \frac{1}{\lambda}\dddot A_{ijk.l}\}
     -(A_{ijl | k}+A_{jkl | i}-A_{kil | j})\nn \\&&+A_{kis}
P^{\ s} _{n \ jl}-A_{jks} P^{\ s} _{n \ il}-A_{ijs} P^{\ s} _{n\
kl}.
  \end{eqnarray}
  From which $P_{njkl}=\frac{1}{\lambda} \ \dddot A_{jkl} +\dot A_{jkl}
$. As $M$ has constant flag curvature we have $\ddot A+\lambda A=0$.
So by the same argument as in  the above theorem we find
\begin{equation}
A(t)=(c_1+c_2 e^{\sqrt{-\lambda}t}+c_3 e^{-\sqrt{-\lambda}t})A(0).
\end{equation}
Using the  boundary assumption on Cartan tensor and letting
$t\rightarrow \infty$ and $t\rightarrow -\infty$, we get
$c_2=c_3=0$. Therefore $A=c_1$ and  $\dot A=0$. It is easy to see
that $A=0$.
\end{proof}
 {\section{Relation between some connections}
 There is a well known  result which can be used as a definition for Landsberg spaces,
 see for example \cite{BCS2}.
 \begin{propriete}
\emph{ Let $(M,F)$ be a Finsler manifold. Then $M$ is a Landsberg
manifold
 if and only if the Berwald connection coincides with the Chern connection.}
\end{propriete}
In this relation we prove the following theorem.
  \begin{thm} Let $(M,F)$ be a complete Finsler manifold with bounded Cartan
 tensor. Then M is a Riemannian manifold if and only if the Berwald connection
 coincides with the Shen connection.\label{thm8}
\end{thm}
 \begin{proof} Simple calculation shows that
  $^b \Gamma^i_{jk}=^s\! \Gamma^i_{jk}+ A^i_{jk}+\dot A^i_{jk}$, where
  $^b\Gamma^i_{jk}$ and $^s\Gamma^i_{jk}$ are the Christoffel coefficients
   of Berwald and  Shen connections respectively. If $^b \Gamma^i_{jk}=^s\!\Gamma^i_{jk}$,
   then
$
A^i_{jk}+\dot A^i_{jk}=0.
 $
By the same argument as in the above theorems, we
 find $A+\dot A=0$ whose  solution is $A(t)=e^{-t}A(0)$. Completeness of $M$
 and the bounded Cartan tensor hypothesis, imply that $A=0$.
 \end{proof}
\setcounter{thm}{0}
\begin{lem}
The Christoffel symbols for  Berwald-type, Cartan-type and Shen-type
connections denoted by $^B\Gamma$, $^C\Gamma$ and $^S\Gamma$
respectively, are given by:
\[
^B\Gamma ^i _{jk} =\frac{g^{is}}{2}\{\frac{\delta g_{sj}}{\delta
x^k}-\frac{\delta g_{jk}}{\delta x^s}+\frac{\delta g_{ks}}{\delta
x^j}\}+{A ^i}_{\ jk}+(\kappa_{_1}\dot A^i_{\ jk} +\cdots+\kappa_{_m}
\overset{_m}{A ^i}_{\ jk}), \quad \quad \quad
\]
\[
^C\Gamma ^i _{jk} =\frac{g^{is}}{2}\{\frac{\delta g_{sj}}{\delta
x^k}-\frac{\delta g_{jk}}{\delta x^s}+\frac{\delta g_{ks}}{\delta
x^j}\}+C^i_{js}N^s_k+(\kappa_{_1}\dot A^i_{\ jk} +\cdots+\kappa_{_m}
\overset{_m}{A ^i}_{\ jk}),\quad \quad \quad
\]
\[
^S\Gamma ^i _{jk} =\frac{g^{is}}{2}\{\frac{\delta g_{sj}}{\delta
x^k}-\frac{\delta g_{jk}}{\delta x^s}+\frac{\delta g_{ks}}{\delta
x^j}\}+(\kappa_{_1}\dot A^i_{\ jk} +\cdots+\kappa_{_m}
\overset{_m}{A ^i}_{\ jk}),\quad \quad \quad \quad \quad \quad \quad
\]
where $\frac{\delta}{\delta x^j}:=\frac{\partial}{\partial
x^j}-N^i_j\frac{\partial}{\partial y^i}$.
\end{lem}
\begin{proof} We prove this lemma for Cartan-type connections only. In the
local coordinate  $(x^i , y^i)$ for $TM_0$, we write $
D_{{{\partial} \over {\partial x^i}}} {\partial_j} =^C\!\!\Gamma ^k
_{ij} {\partial_k}$ and $D_{{{\partial} \over {\partial y^i}}}
{\partial_j} =F^k _{ij} {\partial_k}$. Put $N^k _i = ^C\!\! \Gamma
^k _{ij} y^j=F\{\gamma ^k _{ij} \ell ^j -A^k _{il}
    \gamma ^l _{ab}\ell^a \ell ^b\}$ where  $
   \gamma ^k _{ij}={1 \over 2}g^{kl}\left\{{{\partial {g_{jl}}}\over
   {\partial
x^i}}+ {{\partial {g_{il}}}\over {\partial x^j}} -
    {{\partial {g_{ij}}}\over {\partial x^l}} \right\}
     $.  For Cartan-type connections we consider the compatible tensors $S$ and
     $T$ defined by
 $S=A+\tilde{S}$ and $T=A$, where $\tilde{S}=\kappa_{_{1}}\dot A+\cdots+
 \kappa_{_m}\overset{_m}{A}$. From (\ref{1}) and (\ref{2}) we have
  \begin{eqnarray}
   ^C \Gamma  ^k _{ij}&=&^C \Gamma  ^k _{ji}+N^s_iC^k_{sj}-N^s_jC^k_{si},\label{6-1}\\
  F ^k _{ij}&=&C^k_{ij}+ y^lF^s_{jl}C^k_{si},\label{6-2}\\
{{\partial}\over {\partial x^k}} (g_{ij})&=&g_{li}\ ^C \Gamma  ^l
_{kj}\ -g_{lj}{^C \Gamma ^l} _{ki}\
+2 \tilde{S}_{ijk},\label{6-3}\\
 {{\partial}\over {\partial y^k}} (g_{ij})&=&g_{jl}F^l_{ik}-g_{li}{F^l}_{kj}
. \label{6-4}
\end{eqnarray}
Permuting $i,j$  and $k$ in (\ref{6-3}) and using (\ref{6-1}), one
obtains
\begin{equation}
 ^C \Gamma ^k _{ij} =\gamma ^k _{ij}
+N^s_jC^k_{is}-g^{km}N^s_m C_{ijs}+ \tilde{S}^k _{ij}.\label{6-5}
  \end{equation}
  Since
 $
\frac{g^{is}}{2}\{\frac{\delta g_{sj}}{\delta x^k}-\frac{\delta
g_{jk}}{\delta x^s}+\frac{\delta g_{ks}}{\delta x^j}\}=\gamma ^i
_{jk}-g^{im}N^s_m C_{jks},
$
 we get the desired Christoffel symbols. For other  connections same method  can be used.
\end{proof}

 \setcounter{thm}{2}
\begin{cor}  Let $(M,F)$ be a Finsler manifold. The Berwald-type connection coincides with
 the Shen-type connection if and only if $F$ is Riemannian.
\end{cor}

\textbf{Acknowledgments}. The authors express their sincere thanks to Professor
Zhongmin Shen for his valuable suggestions and comments.

\end{document}